\documentclass[twoside, 12pt]{article}
\usepackage{mathrsfs,amsfonts,amsmath}
\RequirePackage{ifthen,calc}
\usepackage{theorem}
\usepackage{extarrows}
\usepackage[dvips]{color}

\renewcommand{\baselinestretch}{1.0}

 \catcode`@=11
 \def\@evenhead{\hbox to\textwidth{\footnotesize\rm\thepage \hfill
  {\it Li, Y. and Yao Q.}}} 

 \def\@oddhead{\hbox to \textwidth{\footnotesize{\it
Probabilities of deviations for record numbers} \hfill\thepage}}

 \renewcommand{\section}{\makeatletter
 \renewcommand{\@seccntformat}[1]{{\csname the##1\endcsname.}\hspace{0.45em}}
 \makeatother \@startsection
{section}
{1}
{0pt}
{0.25\baselineskip}
{0.25\baselineskip}
{\normalsize\bfseries\mathversion{bold}}}
\catcode`@=11

\newtheorem{theorem}{\noindent Theorem}[section]
\newtheorem{lemma}{\noindent Lemma}[section]
\newtheorem{corollary}{\noindent Corollary}[section]
\newtheorem{remark}{\noindent Remark}[section]

\theoremheaderfont{\normalfont\bfseries}
\theorembodyfont{\slshape} {\theorembodyfont{\rmfamily}}
{\theorembodyfont{\rmfamily} }
{\theorembodyfont{\rmfamily}}
{\theorembodyfont{\rmfamily} }
{\theorembodyfont{\rmfamily} }
{\theorembodyfont{\rmfamily} } {\theorembodyfont{\rmfamily} } {\theorembodyfont{\rmfamily}
}
 \def\beqlb{\begin{eqnarray}}\def\eeqlb{\end{eqnarray}}
 \def\beqnn{\begin{eqnarray*}}\def\eeqnn{\end{eqnarray*}}
 \def\beqls{\begin{equation}}\def\eeqls{\end{equation}}
 \numberwithin{equation}{section}
\def\qed{\hfill$\square$\smallskip}
\def\ZZ{{\mathbb Z}}
\def\bfE{{\mathbb{E}}}
\def\bfP{{\mathbb{P}}}
\def\e{\mathrm{e}}

\def\ii{\mathsf{i}}

\begin{document}

\title{\Large\bf Large and moderate deviations for record numbers in some non--nearest neighbor random walks}
\author{Yuqiang LI\footnote{Key Laboratory of Advanced Theory and Application in Statistics and Data Science--MOE, School of Statistics, East China Normal University.}~~and Qiang YAO\footnote{Corresponding author. Key Laboratory of Advanced Theory and Application in Statistics and Data Science--MOE, School of Statistics, East China Normal University and NYU--ECNU Institute of Mathematical Sciences at NYU Shanghai.
                      E-mail: qyao@sfs.ecnu.edu.cn.}
                      }
\date{}

\maketitle

\begin{abstract}

\noindent The deviation principles of record numbers in random walk models have not been completely investigated, especially for the non--nearest neighbor cases. In this paper, we derive the asymptotic probabilities of large and moderate deviations for the number of ``weak records''~(or ``ladder points'') in two kinds of one--dimensional non--nearest neighbor random walks. The proofs depend only on the direct analysis of random walks. We illustrate that the traditional method of analyzing the local time of Brownian motions, which is often adopted for the simple random walks, would lead to wrong conjectures for our cases.

\smallskip

\noindent {\bf Keywords}:\;{\small Non--nearest neighbor random walks; weak record numbers; large deviations principle; moderate deviations principle}

\smallskip

\noindent{\bf AMS 2000 Subject Classification:}\;{\small 60F10, 60G50}
\end{abstract}


\bigskip

\section{Introduction}

The word ``record'' can be referred as an extreme attainment. The study of record statistics has become indispensable in different fields. In this paper, we are interested in the asymptotic properties of record numbers in random walks as the number of steps tends to infinity, and aim to study the deviations between the record numbers and their asymptotic limits.

\medskip

Let $S=\{S_n,~n\geq 0\}$ be an integer--valued random walk on $\ZZ$~(may be non--nearest neighbor), namely, $S_{_0}=0$, and $S_n=\sum\limits_{i=1}^n X_i$ for $n\geq 1$,
where $X_{_1}$, $X_{_2}$, $\cdots$ are i.i.d. integer--valued random variables. Define $M_n=\max\limits_{0\leq m\leq n}S_m$ for $n\geq1$.
Let $T_0=0$, $T_{n}=\inf\{m>T_{n-1}, S_m\geq M_{m-1}\}~\text{for}~n\geq1$,
and define
 \beqlb\label{def-2}
 A_n=\sup\{k\geq 1, \quad T_k\leq n\}
 \eeqlb
for each $n\geq 1$, where $\inf\emptyset\stackrel{\text{def}}{=}+\infty$ and $\sup\emptyset\stackrel{\text{def}}{=}0$. In this paper, we call $A_n$ the \emph{weak} record numbers up to time $n$. The word ``weak'' is to emphasize that we not only consider the time when a new record appears, but also keep eyes on the time when the current record is repeated. It is important to note that the weak record number we are considering now is different from the ``record numbers'' studied in Katzenbeisser and Panny \cite{KP92},  Kirschenhofer and Prodinger \cite{KP96}, P\u{a}tt\u{a}nea \cite{P2007}, where they discussed the number of the events $\{S_k=M_n\}$~(rather than $\{S_k=M_k\}$) that occur up to time $n$.

\medskip

In some literatures, $A_n$ is also called the number of ``weak ladder points'' which is a footstone in the fluctuation theory of random walks. The fluctuation theory was first proposed by Spitzer \cite{sp64} and Feller \cite{Feller71}, and has drawn much attention since then because of its wide applications. For more details, one can refer to Karlin and Taylor \cite[Chapter 17]{KT81}.  Omey and Teugels \cite{GOT82} proved that a normalized version of the bivariate ladder process $\{(T_n, S_{T_n})\}$ converges in law to the bivariate ladder process of a L\'{e}vy process $X$ whenever the normalized $\{S_n\}$ converges in law to $X$. As an immediate corollary, one can derive that a normalized version of $A_n$~(number of ladder points) of $S$ converges in distribution to the local time at the supremum of $X$. Later, Chaumont and Doney \cite{CD2010} extended this result to a more general case. Based on the above results, one may further ask about the deviations between the normalized version of $A_n$ and its limit. As far as we know, there is very few related research to investigate such problems.

\medskip

In this paper, we study the asymptotic probabilities of $\bfP(A_n\geq \sqrt{n}c_n)$, where $c_n$ tends to infinity under some constraints. We will establish the large deviations principle (LDP) and moderate deviations principle (MDP) for $A_n$, respectively. For a general theory of LDP and MDP, please refer to Dembo and Zeitouni \cite{DZ98}.

\medskip

Let $Y_k=T_{k}-T_{k-1}$ for $k\geq 1$. The strong Markov property of random walks implies that $Y_k$'s are i.i.d, and $A_n=\sup\left\{k, \sum\limits_{i=1}^k Y_i\leq n\right\}$. Namely, $\{A_n\}_{n\geq 1}$ is a discrete time renewal process with the inter--occurrence time sequence $\{Y_n\}$.
There are many results on the theory of deviations for renewal processes or renewal reward processes. See, for example, Serfozo \cite{SR74}, Glynn and Whitt \cite{GW94}, Jiang \cite{J1994}, Chi \cite{C2007}, Lefevere et al. \cite{LMZ11}, Borovkov and Mogulskii \cite{BM2015}, Tsirelson \cite{T2013}, Logachov and Mogulskii \cite{LM2018}, and the references therein. However, these proposed approaches cannot be applied directly to our case, since most of them require constraints on moments or moment generating functions for inter-occurrence times, which are not fulfilled by $A_n$ in our situations.

\medskip

By adopting the celebrated invariance principle, one may naturally connect $A_n$ with the maximal value process $B^*(n)$ of a Brownian motion when the increments of random walk $S$ have finite variance, and conceive that we can get the LDP and MDP of $A_n$ by extending the asymptotical results for $2B^*(n)$. However, our results (see Theorem \ref{ldp} and Corollary \ref{thmmdp-2}) show that this method would lead to wrong conjectures. Instead, in this paper, we investigate the LDP and MDP for $A_n$ via the deviation theory of occupation time of Markov process as well as some analysis on the related queueing models. This is our main contribution.

\medskip

The remainder of this paper is organized as follows. In Section 2, we summarize the main results of this paper, and highlight our main contributions.
In Section 3, we provide some results for queueing models, which are crucial for the analysis of left or right continuous random walks. Then we establish the LDP in Section 4 and the MDP in Section 5. Finally, we make some concluding remarks in Section 6.

\section{Statement of main results}

Let $S$ be the random walk as defined in Section 1. We say $S$ is \emph{right continuous} if the probability mass function (p.m.f) of $X_i$ satisfies
\beqlb\label{rc-0}
0<q=\bfP(X_i=1),\qquad \bfP(X_i=-n)=p_n, \;\; n\geq 0.
\eeqlb
Similarly, we say $S$ is \emph{left continuous} if the p.m.f of $X_i$ satisfies
\beqlb\label{lc-0}
0<q=\bfP(X_i=-1),\qquad \bfP(X_i=n)=p_n, \;\; n\geq 0.
\eeqlb
The notions of ``right continuous'' and ``left continuous'' first appeared in Spitzer \cite{sp64}. Let $$\phi(s)=q+\sum\limits_{n=0}^\infty p_n s^{n+1},$$
for $s\in [0, 1]$. For convenience, in the sequel, we call the random walk $S$ is right or left continuous with $\phi$ if the p.m.f of its increments has the form of (\ref{rc-0}) or (\ref{lc-0}), respectively.

Obviously, for each $s\in[0, 1]$, the equation $x=s\phi(x)$ has a solution $x_s\in[0, 1]$. We denote the minimum non-negative solution by $h(s)$, which will be discussed in more details in Lemma 3.1 later. For every $\lambda\in(-\infty, 0]$, let
 $$\Lambda_r(\lambda)=\ln\left(1+q\e^\lambda-\frac{q\e^\lambda}{h(\e^\lambda)}\right)~~\text{and}~~\Lambda_l(\lambda)=\lambda+\ln\left(\frac{1-\phi(h(\e^\lambda))}{1-h(\e^\lambda)}\right).$$
As shown in Lemma \ref{lemlambda} in the next section, $\Lambda^\prime_r(\lambda)$ and $\Lambda^\prime_l(\lambda)$ are  continuous monotone functions with  $\Lambda_r^\prime(0)=\Lambda_l^\prime(0)=+\infty$ and $\Lambda_r^\prime(-\infty)=\Lambda_l^\prime(-\infty)=1$. Therefore, for any $x\in(1, +\infty)$, there exist unique $\lambda_l,\lambda_r\in[-\infty,0)$ such that $x=\Lambda_r^\prime(\lambda_r)$ and $x=\Lambda_l^\prime(\lambda_l)$. Denote  $\lambda_l$, $\lambda_r$ by $G_l(x)$ and $G_r(x)$, respectively. For each $x\geq 0$,  define
 $$\Lambda_r^*(x)=\sup_{\lambda\leq 0}\{x\lambda-\Lambda_r(\lambda)\}=\begin{cases} xG_r(x)-\Lambda_r(G_r(x)), \quad & x>1,
\\ -\ln (q+p_{_0}), & x=1,
\\ +\infty, & x<1.
\end{cases}$$
and
$$\Lambda_l^*(x)=\sup_{\lambda\leq 0}\{x\lambda-\Lambda_l(\lambda)\}=\begin{cases} xG_l(x)-\Lambda_l(G_l(x)), \quad & x>1,
\\ -\ln (1-q),  & x=1,
\\ +\infty, & x<1.
\end{cases}$$

Let $A_n$ be the weak record number of $S$ up to time $n$, which is defined by (\ref{def-2}). We have the following LDP for $A_n$.


\begin{theorem}\label{ldp}  Suppose $\phi^\prime(1)=1$. For any $x>0$, if $S$ is right continuous with $\phi$, then
 \beqnn
 \lim_{n\to +\infty}\frac{1}{n}\ln\bfP(A_n\geq x n)= -x\Lambda_r^*(1/x).
 \eeqnn
If $S$ is left continuous with $\phi$, then
 \beqnn
 \lim_{n\to +\infty}\frac{1}{n}\ln\bfP(A_n\geq x n)= -x\Lambda_l^*(1/x).
 \eeqnn
\end{theorem}

\medskip

To facilitate our discussion in MDP, we need the following technical assumption.

\noindent{\bf Assumption} (H): There exist $\alpha\in(0,1)$ and $c>0$  such that $\displaystyle{\lim_{s\to 1-} \frac{1-s\phi^\prime(h(s))}{(1-s)^\alpha}=c}$.

\medskip

The MDP for $A_n$ is as follows.

\begin{theorem}\label{thmmdp-1} Suppose $\phi^\prime(1)=1$ and Assumption (H) holds. Let $\{c_n\}$ be a sequence of positive numbers such that $c_n\to+\infty$ and $c_n=o(n)$.
\begin{itemize}
\item[(1)] If $S$ is right continuous with $\phi$, then for any $x>0$,
 \beqnn
 \lim_{n\to +\infty}\frac{1}{c_n}\ln\bfP(A_n\geq x {n}^{1-\alpha}c_n^{\alpha})= -\frac{\alpha}{1-\alpha}\left(\frac{q}{c}\right)^{1/\alpha}x^{1/\alpha}.
 \eeqnn
\item[(2)] If $S$ is left continuous with $\phi$, then for any $x>0$,
 \beqnn
 \lim_{n\to +\infty}\frac{1}{c_n}\ln\bfP(A_n\geq x {n}^{\alpha}c_n^{1-\alpha})= -\left[c(1-\alpha)^{2-\alpha}\alpha^{\alpha}\right]^{1/(1-\alpha)}x^{1/(1-\alpha)}.
 \eeqnn
 \end{itemize}
\end{theorem}

By applying Theorem \ref{thmmdp-1} to some special cases, we obtain the following corollaries.

\begin{corollary}\label{thmmdp-2} Suppose $\phi^\prime(1)=1$ and $\sigma:=\sqrt{\phi^{\prime\prime}(1)}<+\infty$. Let $\{c_n\}$ be a sequence of positive numbers such that $c_n\to+\infty$ and $c_n=o(n)$ as $n$ tends to infinity.
\begin{itemize}
\item[(1)] If $S$ is right continuous with $\phi$, then for any $x>0$,
 \beqnn
 \lim_{n\to +\infty}\frac{1}{c_n}\ln\bfP(A_n\geq x {n}^{1/2}c_n^{1/2})= -\frac{q^2x^2}{2\sigma^2}.
 \eeqnn
\item[(2)] If $S$ is left continuous with $\phi$, then for any $x>0$,
 \beqnn
 \lim_{n\to +\infty}\frac{1}{c_n}\ln\bfP(A_n\geq x {n}^{1/2}c_n^{1/2})= -\frac{\sigma^2}{8}x^2.
 \eeqnn
 \end{itemize}
\end{corollary}

To investigate the case of $\phi^{\prime\prime}(1)=+\infty$, we satisfy ourselves by studying the special case of $\phi(s)=s+\frac{\gamma}{1+\beta}(1-s)^{1+\beta}$.
\begin{corollary}\label{thmmdp-3} Suppose $\phi(s)=s+\frac{\gamma}{1+\beta}(1-s)^{1+\beta}$ where $\gamma,\beta\in (0, 1)$.  Let $\{c_n\}$ be a sequence of positive numbers such that $c_n\to+\infty$ and $c_n=o(n)$ as $n$ tends to infinity.
\begin{itemize}
 \item[(1)] If $S$ is right continuous with $\phi$, then for any $x>0$,
 \beqnn
 \lim_{n\to +\infty}\frac{1}{c_n}\ln\bfP(A_n\geq x {n}^{1/(1+\beta)}c_n^{\beta/(1+\beta)})= -\frac{\beta\gamma}{(1+\beta)^{2+1/\beta}}x^{1+1/\beta}.
 \eeqnn
\item[(2)] If $S$ is left continuous with $\phi$, then for any $x>0$,
 \beqnn
 \lim_{n\to +\infty}\frac{1}{c_n}\ln\bfP(A_n\geq x {n}^{\beta/(1+\beta)}c_n^{1/(1+\beta)})= -\frac{\gamma \beta^{\beta}}{(1+\beta)^{2+\beta}}x^{1+\beta}.
 \eeqnn
 \end{itemize}
\end{corollary}

\begin{remark}
When $\phi^\prime(1)=1$ and $\sigma^2=\phi^{\prime\prime}(1)<+\infty$, the expectation and the variance of $X_i$ are $0$ and $\sigma^2$, respectively. In this case, by the strong invariance principle, $S$ is approximated by a Brownian motion with variance parameter $\sigma^2$, whether $S$ is right continuous or left continuous. However, as indicated in Theorem \ref{ldp} and Corollary \ref{thmmdp-2}, the right or left continuity of random walk $S$ leads to different rate functions for the LDP and MDP of $A_n$. These observations show that for the problems investigated here, it would lead to wrong conjectures by simply extending the asymptotic results of Brownian motions to random walks via the invariance principle.
\end{remark}
\begin{remark}
It is well known that if $Y_k$'s are i.i.d. with the same probability generating functions $\phi(s)$, then as $n\rightarrow\infty$,
   $$\bfE\left(\e^{\ii t\sum\limits_{k=1}^n (Y_k-1)/n^{1/(1+\beta)}}\right)\to\exp\left\{-\frac{|\cos((1+\beta)\pi/2)|\gamma}{1+\beta}|t|^{1+\beta}\left(1-\ii{\rm sgn}(t)\tan\left(\frac{(1+\beta)\pi}{2}\right)\right)\right\},$$
 which is the characteristic function of a $(1+\beta)$-stable distribution, saying $U$, without negative jumps. Therefore the distribution determined by $\phi(s)$ belongs to the domain of attraction of the stable distribution $U$. Furthermore, as
shown by Skorohod~(1957), $$\frac{1}{n^{1/(1+\beta)}}\sum_{k=1}^{\lfloor nt\rfloor}(Y_k-1)$$ converges weakly in the Skorohod space $D([0,1])$ with $J_1$ topology to a L\'{e}vy stable motion $L(t)$ whose distribution at $t=1$ is $U$, where $\lfloor a\rfloor$ denotes the
maximal integer no larger than $a$. As a result, if $S$ is left or right continuous with $\phi$, then $S_{\lfloor nt\rfloor}/n^{1/(1+\beta)}$ converges weakly in $D([0,1])$ to $L(t)$ or $-L(t)$, respectively.
\end{remark}

\section{Some results for queueing models}

Our main approach in analyzing the property of left continuous and right continuous random walks is to relate them to the queueing models.

\medskip

Let $p_{_{-1}}, p_{_0}, p_{_1},\cdots$ be a sequence of non-negative real numbers such that $\sum\limits_{n=-1}^{+\infty} p_n=1$ and $p_{_0}<1$. Let $W=\{W_n; n\geq 0\}$ be a Markov chain with transition probabilities $(p_{i,j})_{i,j\geq 0}$, where
 \beqlb\label{rc-1}
 p_{i,j}=\begin{cases} p_k, \quad &j=i+k, k\geq -1, i\geq 1;
\\ p_k, & j=k, i=0, k\geq 1;
\\ p_{_0}+p_{_{-1}}, & j=i=0;
\\ 0, &\text{otherwise}.
\end{cases}
\eeqlb
or
\beqlb\label{lc-1}
 p_{i,j}=\begin{cases} p_{i-j}, \quad &i\geq 1, 0<j\leq i;
\\ \sum\limits_{k=i}^{+\infty} p_k, & i\geq 0, j=0;
\\ p_{_{-1}}, & j=i+1, i\geq 0;
\\ 0, &\text{otherwise}.
\end{cases}
\eeqlb

Intuitively, in a service system with one server, if $p_{i,j}$ has the form of (\ref{lc-1}), $W$ is the length of the waiting line when a new customer enters the service system, where $p_k$ denotes the probability of exactly $k+1$ customers served in an inter-arrival period. If $p_{i,j}$ has the form of (\ref{rc-1}), $W$ is the length of the waiting line~(excluding the customer in service) when a customer leaves the service system, where $q_k$ is the probability of exactly $k+1$ customers arriving in a service period.

The following lemmas are crucial to the proof of our main results. They are of independent interest as well. Although they seem to be some fundamental conclusions for the process $W$, we cannot find a suitable reference. For the convenience of reference, we provide their detailed proofs in the following. Let $\phi(s)=\sum\limits_{n=-1}^{+\infty} p_n s^{n+1}$ for $s\in[0, 1]$. Then we have the following result.
\begin{lemma}\label{lemma1}
If $\phi^\prime(1)=1$, then there exists a unique differentiable function $h(s)\in[0, 1]$ such that $h(s)=s\phi(h(s))$ for all $s\in[0, 1]$. Furthermore, $h(s)$ has the following properties:
\begin{itemize}
\item[(1)] $h(s)/s\to \phi(0)=p_{_{-1}}$ as $s\to 0+$, and $h(s)\to 1$ as $s\to 1-$;
\item[(2)] $h^\prime(s)=\frac{\phi(h(s))}{1-s\phi^\prime(h(s))}>0$ for all $s\in(0, 1)$;
\item[(3)] $(h(s)-p_{_{-1}}s)/s^2\to p_{_0}p_{_{-1}}$ as $s\to 0+$;
\item[(4)] if $\sqrt{\phi^{\prime\prime}(1)}=\sigma<+\infty$, then $\lim\limits_{s\to 1-} \frac{1-s\phi^\prime(h(s))}{\sqrt{1-s}}=\sqrt{2}\sigma$.
\end{itemize}
\end{lemma}
\begin{proof}  Applying the intermediate value theorem to the function $x-s\phi(x)$ with the variable $x$ and noting its monotonicity in $[0, 1]$, we can readily know that there exists a unique function $h(s)\in [0, 1]$ such that $h(s)=s\phi(h(s))$ for all $s\in[0, 1]$.  As a result, the implicit function theorem leads to (1) and (2). Next, we give the detailed proof for (3) and (4).

 \medskip

 To prove (3), we use the L'Hospital rule and the formula of $h^\prime(s)$ in (2) to obtain that
 \beqnn
 \lim_{s\to 0+}\frac{h(s)-p_{_{-1}}s}{s^2}&=&\lim_{s\to 0+}\frac{h^\prime(s)-p_{_{-1}}}{2s}=\lim_{s\to 0+}\frac{\phi(h(s))-p_{_{-1}}(1-s\phi^\prime(h(s)))}{2(1-s\phi^\prime(h(s)))s}
 \\&=&\lim_{s\to 0+}\frac{\phi^\prime(h(s))h^\prime(s)+p_{_{-1}}\phi^\prime(h(s))+p_{_{-1}}s\phi^{\prime\prime}(h(s))h^\prime(s)}{2(1-s\phi^\prime(h(s)))-2s(\phi^\prime(h(s))+s\phi^{\prime\prime}(h(s))h^\prime(s))}.
 \eeqnn
 Note that as $s\to0$, $h(s)\to 0$, $\phi^\prime(s)\to p_{_0}$ and $h^\prime(s)\to p_{_{-1}}$. We obtain that
 \beqnn
 \lim_{s\to 0+}\frac{h(s)-p_{_{-1}}s}{s^2}&=&\lim_{s\to 0+}\frac{\phi^\prime(h(s))h^\prime(s)+p_{_{-1}}\phi^\prime(h(s))}{2}=p_{_{-1}}p_{_0}.
 \eeqnn

 \medskip

 To prove (4), by the L'Hospital rule again,
 \beqnn
 \lim_{s\to 1-} \frac{(1-s\phi^\prime(h(s)))^2}{1-s}&=&\lim_{s\to 1-} 2(1-s\phi^\prime(h(s)))(\phi^\prime(h(s))+s\phi^{\prime\prime}(h(s))h^\prime(s))
 \\&=&\lim_{s\to 1-} 2(1-s\phi^\prime(h(s)))\phi^{\prime\prime}(h(s))h^\prime(s).
 \eeqnn
From (2), we have that $\phi(h(s))=(1-s\phi^\prime(h(s)))h^\prime(s)$. Therefore,
 \beqnn
 \lim_{s\to 1-} \frac{(1-s\phi^\prime(h(s)))^2}{1-s}=\lim_{s\to 1-} 2\phi(h(s))\phi^{\prime\prime}(h(s))=2\sigma^2,
 \eeqnn
which implies the desired result.\qed
\end{proof}

\medskip

Let $\tau=\inf\{n>0, W_n=0\}$, and define $f_k(s)=\bfE(s^{\tau}|W_{_0}=k)$.

\medskip

\begin{lemma}\label{lemrc-1}
Suppose that for each pair $(i,j)$, the transition probability $p_{i,j}$ is given by (\ref{rc-1}). Then $\displaystyle{f_{_0}(s)=1+p_{_{-1}}s-\frac{p_{_{-1}}s}{h(s)}}$.
\end{lemma}
\begin{proof} By the one step analysis of Markov chains, we know that the family of functions $\{f_k(s); k\geq 0\}$ satisfies the following equations,
 \beqlb
 f_{_0}(s)&=&s\left(\sum_{k=1}^{+\infty} p_k f_k(s)+p_{_0}+p_{_{-1}}\right),\label{rc-2}
\\ f_{_1}(s)&=&s\left(\sum_{k=1}^{+\infty} p_{k-1} f_k(s)+p_{_{-1}}\right).\label{rc-3}
 \eeqlb
For $k\geq1$, since $W$ is left continuous, when $W_0=k$, we have $\tau=\sum\limits_{i=1}^k\tau_i$, where $\tau_i$ is the first time that $W$ hits $i-1$ starting from $i$~($i=1,\cdots,k$). So by the Markov and the homogeneous property of $W$, we have $$f_k(s)=\prod\limits_{i=1}^kE\left(\left.s^{\tau_i}\right|W_0=i\right)=\prod\limits_{i=1}^kE\left(\left.s^\tau\right|W_0=1\right)=(f_1(s))^k$$ for $k\geq1$, together with (\ref{rc-3}), it implies $f_{_1}(s)=h(s)$. Therefore, from (\ref{rc-2}), we obtain
 \beqnn
 f_{_0}(s)&=&s\left(\sum_{k=1}^{+\infty} p_k h^{k+1}(s)+(p_{_0}+p_{_{-1}})h(s)\right)/h(s)
 \\&=&s\left(\phi(h(s))-p_{_{-1}}+p_{_{-1}}h(s)\right)/h(s)=1+p_{_{-1}}s-\frac{p_{_{-1}}s}{h(s)},
 \eeqnn
as desired.\qed
\end{proof}

\medskip

\begin{lemma}\label{lemlc-1}
Suppose that for each pair $(i,j)$, the transition probability $p_{i,j}$ is given by (\ref{lc-1}). Then
 $\displaystyle{f_{_0}(s)=\frac{s(1-\phi(h(s)))}{(1-h(s))}}$.
\end{lemma}

\begin{proof} By the one step analysis of Markov chains, we know that $\{f_k(s); k\geq 0\}$ satisfies
 \beqlb
 f_k(s)&=&s\sum_{n=k}^{+\infty} p_n+s\sum_{j=-1}^{k-1}p_jf_{k-j}(s),\qquad k\geq 0. \label{lc-2}
 \eeqlb
For every $(u,s)\in[0, 1]\times[0, 1]$, define $F(u,s)=\sum\limits_{k=0}^{+\infty} u^kf_k(s)$. From (\ref{lc-2}), we obtain
 \begin{align*}
 &F(u,s)=\sum_{k=0}^{+\infty} u^k\left(s\sum_{n=k}^{+\infty} p_n+s\sum_{j=-1}^{k-1}p_jf_{k-j}(s)\right)
 \\=&s\sum_{n=0}^{+\infty} p_n\sum_{k=0}^n u^k+s\sum_{j=-1}^{+\infty} p_j\sum_{k=j+1}^{+\infty}u^kf_{k-j}(s)=s\frac{1-\phi(u)}{1-u}+\frac{s}{u}\phi(u)(F(u,s)-f_{_0}(s)),
 \end{align*}
which leads to $\displaystyle{F(u,s)=\frac{su(1-\phi(u))-s(1-u)\phi(u)f_{_0}(s)}{(1-u)(u-s\phi(u))}}$.
For each $s\in(0, 1)$, let $$D_+(s)=\{u\in[0, 1];\; u-s\phi(u)>0 \}\;\;\text{and}\;\; D_-(s)=\{u\in[0, 1];\; u-s\phi(u)<0\}.$$ Since $t-s\phi(t)$ is non--decreasing for $t\in[0, 1]$, we know that $u<h(s)<v$ for each $u\in D_-(s)$ and $v\in D_+(s)$. In addition, due to the fact that $F(u, s)\geq 0$ for all $(u, s)\in[0, 1]\times[0, 1]$, we have that $sv(1-\phi(v))-s(1-v)\phi(v)f_{_0}(s)\geq 0$ for all $v\in D_+(s)$, which means that $\displaystyle{f_{_0}(s)\leq\frac{v(1-\phi(v))}{(1-v)\phi(v)}}$ for all $v\in D_+(s)$. Similarly, $\displaystyle{f_{_0}(s)\geq\frac{u(1-\phi(u))}{(1-u)\phi(u)}}$ for all $u\in D_-(s)$. Observe that $\displaystyle{\left(\frac{s(1-\phi(s))}{(1-s)\phi(s)}\right)^\prime=\frac{\phi(s)-\phi^2(s)-s(1-s)\phi^\prime(s)}{(1-s)^2\phi^2(s)}}$ for each $s\in(0, 1)$. From the facts that $\phi^\prime(1)=1$, $p_{_0}<1$, as well as the convexity of $\phi$, we know that  $p_{_{-1}}=\phi(0)>0$, $\phi(s)>s$ and
   $$\phi(s)-\phi^2(s)-s(1-s)\phi^\prime(s)>s(1-\phi(s)-(1-s)\phi^\prime(s))>0$$ for all $s\in(0, 1)$. Therefore, the function $\dfrac{s(1-\phi(s))}{(1-s)\phi(s)}$ is non--decreasing for $s\in[0, 1]$, which implies that
 $$\lim_{u\to h(s)-}\frac{u(1-\phi(u))}{(1-u)\phi(u)}\leq f_{_0}(s)\leq \lim_{v\to h(s)+} \frac{v(1-\phi(v))}{(1-v)\phi(v)}.$$
That is, $\displaystyle{f_{_0}(s)=\frac{h(s)(1-\phi(h(s)))}{(1-h(s))\phi(h(s))}}$, which leads to the desired result.\qed
\end{proof}

\medskip

For each $\lambda<0$, let $\Lambda(\lambda)=\ln f_{_0}(\e^\lambda)$. Then we have the following result.

\begin{lemma}\label{lemlambda}
    If (\ref{rc-1}) or (\ref{lc-1}) holds, then $\Lambda^\prime(\lambda)$ is increasing, and
      $$\lim_{\lambda\to-\infty}\Lambda^\prime(\lambda)=1,\qquad\lim_{\lambda\to 0-}\Lambda^\prime(\lambda)=+\infty.$$
 \end{lemma}
\begin{proof} From the fact that $\Lambda(\lambda)=\ln(f_{_0}(\e^\lambda))=\ln(\bfE(\e^{\lambda \tau}|W_{_0}=0))$, we have
 $$\Lambda^{\prime\prime}(\lambda)=\frac{\bfE(\tau^2\e^{\lambda \tau}|W_{_0}=0)\bfE(\e^{\lambda \tau}|W_{_0}=0)-[\bfE(\tau\e^{\lambda \tau}|W_{_0}=0)]^2}{[\bfE(\tau\e^{\lambda \tau}|W_{_0}=0)]^2}.$$
Therefore, the H\"{o}lder's inequality implies that $\Lambda^{\prime\prime}(\lambda)>0$ for all $\lambda<0$ and hence $\Lambda^\prime(\lambda)$ is monotone increasing for $\lambda\in(-\infty, 0)$.

\medskip

To show the limits, we first consider the case when (\ref{rc-1}) holds. In this case, from Lemma \ref{lemrc-1}, we have
\beqnn
\Lambda(\lambda)=\ln\left(1+p_{_{-1}}\e^{\lambda}-\frac{p_{_{-1}}\e^\lambda}{h(\e^\lambda)}\right).
\eeqnn
By simple computations, we obtain $\displaystyle{\Lambda^\prime(\lambda)=\frac{p_{_{-1}}\e^\lambda\left(h(\e^\lambda)\left(1-\e^\lambda\phi^\prime(h(\e^\lambda))\right)+\e^\lambda\phi^\prime\left(h(\e^\lambda)\right)\right)}{\left(1-\e^\lambda\phi^\prime(h(\e^\lambda))\right)\left[(1+p_{_{-1}}\e^\lambda)h(\e^\lambda)-p_{_{-1}}\e^\lambda\right]}}$. Using Lemma \ref{lemma1} and noting that $\phi(0)=p_{_{-1}}$, $\phi^\prime(0)=p_{_0}$ and $\phi(1)=\phi^\prime(1)=1$, we get
\beqnn
\lim_{\lambda\to 0-}\Lambda^\prime(\lambda)&=&\lim_{\lambda\to 0-}\frac{p_{_{-1}}}{1-\e^\lambda\phi^\prime(h(\e^\lambda))}=+\infty,
\\ \lim_{\lambda\to -\infty}\Lambda^\prime(\lambda)&=&\lim_{\lambda\to -\infty}\frac{p_{_{-1}}\e^\lambda\left(h(\e^\lambda)+\e^\lambda\phi^\prime\left(h(\e^\lambda)\right)\right)}{(1+p_{_{-1}}\e^\lambda)h(\e^\lambda)-p_{_{-1}}\e^\lambda}=\lim_{\lambda\to -\infty}\frac{p_{_{-1}}(p_{_{-1}}+p_{_0})\e^{2\lambda}}{p_{_{-1}}(p_{_{-1}}+p_{_0})\e^{2\lambda}}=1.
\eeqnn

\medskip

Next, we consider the case when (\ref{lc-1}) holds.  From Lemma \ref{lemlc-1}, it follows that $\displaystyle{\Lambda(\lambda)=\ln\left(\e^\lambda\frac{1-\phi(h(\e^\lambda))}{1-h(\e^\lambda)}\right)}$. It is easy to see that $\displaystyle{\Lambda^\prime(\lambda)=1+\frac{h^\prime(\e^\lambda)\e^\lambda}{1-h(\e^\lambda)}-\frac{\phi^\prime(h(\e^\lambda))h^\prime(e^\lambda)\e^\lambda}{1-\phi(h(\e^\lambda))}}$. Using Lemma \ref{lemma1} and noting that $\phi(0)=p_{_{-1}}$, $\phi^\prime(0)=p_{_0}$ and $\phi(1)=\phi^\prime(1)=1$, we obtain
$$
\lim_{\lambda\to -\infty}\Lambda^\prime(\lambda)=1+\lim_{s\to 0+}\left[\frac{h^\prime(s)s}{1-h(s)}-\frac{\phi^\prime(h(s))h^\prime(s)s}{1-\phi(h(s))}\right]=1,$$
and that
\beqnn
\lim_{\lambda\to 0-}\Lambda^\prime(\lambda)&=&\lim_{s\to 1-}\left[\frac{1-\phi(h(s))-\phi^\prime(h(s))(1-h(s))}{(1-h(s))(1-\phi(h(s)))}h^\prime(s)s+1\right].
\eeqnn
From Lemma \ref{lemma1}, we know that as $s\to 1$, $h(s)\to 1$ and $h^\prime(s)\to +\infty$. It is easy to get
 $$\lim_{s\to 1-}\frac{1-\phi(h(s))-\phi^\prime(h(s))(1-h(s))}{(1-h(s))(1-\phi(h(s)))}h^\prime(s)=\infty.$$
Consequently, $\lim\limits_{\lambda\to 0-}\Lambda^\prime(\lambda)=\infty.$\qed
\end{proof}

\begin{lemma}\label{lemlimit}
    If (\ref{rc-1}) holds, then $$\lambda-\Lambda(\lambda)\to -\ln(p_{_{-1}}+p_{_0})~~\text{as}~~\lambda\to-\infty.$$
    If (\ref{lc-1}) holds, then $$\lambda-\Lambda(\lambda)\to -\ln(1-p_{_{-1}})~~\text{as}~~\lambda\to-\infty.$$
\end{lemma}

\begin{proof} When (\ref{rc-1}) holds,
 $\displaystyle{\lambda-\Lambda(\lambda)=\ln\left(\frac{\e^\lambda h(\e^\lambda)}{(1+p_{_{-1}}\e^\lambda)h(\e^\lambda)-p_{_{-1}}\e^\lambda}\right)}$.
Then
 $$
 \lim_{\lambda\to-\infty}\lambda-\Lambda(\lambda)=\lim_{\lambda\to-\infty}\ln\left(\frac{\e^\lambda h(\e^\lambda)}{(1+p_{_{-1}}\e^\lambda)h(\e^\lambda)-p_{_{-1}}\e^\lambda}\right)=\ln\left(\lim_{s\to 0+}\frac{s h(s)}{(1+p_{_{-1}}s)h(s)-p_{_{-1}}s}\right).
 $$
Using (3) in Lemma \ref{lemma1}, we obtain
  $$
 \lim_{\lambda\to-\infty}\lambda-\Lambda(\lambda)=\ln\left(\lim_{s\to 0+}\frac{p_{_{-1}}s^2}{p_{_{-1}}s+p_{_{-1}}p_{_0}s^2+p_{_{-1}}^2s^2-p_{_{-1}}s}\right)=\ln\left(\frac{1}{p_{_{-1}}+p_{_0}}\right)=-\ln(p_{_{-1}}+p_{_0}).
 $$

When (\ref{lc-1}) holds, we have
 \beqnn
 \lim_{\lambda\to-\infty}\lambda-\Lambda(\lambda)&=&\lim_{\lambda\to-\infty}\ln\left(\frac{1-h(\e^\lambda)}{1-\phi(h(\e^\lambda))}\right)=\ln\left(\lim_{s\to 0+}\frac{1-s}{1-\phi(s)}\right)=-\ln(1-p_{_{-1}}).
 \eeqnn
The proof is completed.\qed
\end{proof}

\begin{lemma}\label{mdplem-1}
Let $P(s)=1/(1-f_{_0}(s))$ for $s\in(0, 1)$. Suppose that there exists $\alpha\in(0,1)$ and $c>0$  such that $$\lim_{s\to 1-} \frac{1-s\phi^\prime(h(s))}{(1-s)^\alpha}=c.$$ Then when (\ref{rc-1}) holds,
 $$\lim_{s\to 1-}P(s)(1-s)^{1-\alpha}=\frac{(1-\alpha)c}{p_{_{-1}}},$$
and when (\ref{lc-1}) holds,
 $$\lim_{s\to 1-}P(s)(1-s)^{\alpha}=\frac{1}{(1-\alpha)c}.$$
\end{lemma}
\begin{proof} If (\ref{rc-1}) holds, then by Lemma \ref{lemma1} (2), Lemma \ref{lemrc-1} and Assumption (H),
 \begin{align*}
 &\lim_{s\to 1-}P(s)(1-s)^{1-\alpha}=\lim_{s\to 1-}\frac{h(s)(1-s)^{1-\alpha}}{p_{_{-1}}s(1-h(s))}=\lim_{s\to 1-}\frac{(1-s)^{1-\alpha}}{p_{_{-1}}(1-h(s))}\\
 =&\lim_{s\to 1-}\frac{1-\alpha}{p_{_{-1}}(1-s)^\alpha h^\prime(s)}=\frac{1-\alpha}{p_{_{-1}}}\lim_{s\to 1-}\frac{1-s\phi^\prime(h(s))}{(1-s)^\alpha\phi(h(s))}=\frac{c(1-\alpha)}{p_{_{-1}}}.
 \end{align*}

If (\ref{lc-1}) holds, then  Lemma \ref{lemlc-1} implies that
 \beqnn
P(s)=\frac{1}{1-f_{_0}(s)}=\frac{1-h(s)}{1-h(s)-s+s\phi(h(s))}=\frac{1-h(s)}{1-s}.
\eeqnn
Therefore, from Assumption (H), we obtain
 \beqnn
\lim_{s\to 1-}P(s)(1-s)^{\alpha}=\lim_{s\to 1-}\frac{1-h(s)}{(1-s)^{1-\alpha}}=\frac{1}{(1-\alpha)c}.
\eeqnn
The proof is completed.\qed
 \end{proof}

\section{The proof of LDP}
In this section, we will provide the proof of LDP. Let $\bar{S}_{_0}=0$ and $\bar{S}_n=M_n-S_n$ for $n\geq 1$, where $\{S_n\}$ is the random walk given in Section 2, and $M_n=\max\limits_{0\leq k\leq n} S_k$. For any $n\geq 0$,
\beqnn
 \bar{S}_{n+1}&=&M_{n+1}-S_{n+1}=S_{n+1}\vee M_n-S_{n+1}
\\&=&(S_{n+1}-S_n)\vee (M_n-S_n)+S_n-S_{n+1}=(S_{n+1}-S_n)\vee \bar{S}_n-(S_{n+1}-S_n),
\eeqnn
Since $S_{n+1}-S_n$ is independent of $\{S_k, 0\leq k\leq n\}$ and has the same distribution,
$\{\bar{S}_n, n\geq 0\}$ is a nonnegative time-homogeneous Markov chain with one-step transition probabilities
 \beqnn
p_{i,j}=\bfP(\bar{S}_{n+1}=j|\bar{S}_n=i)=\begin{cases} \bfP(S_{n+1}-S_n=i-j), & j>0,
\\ \bfP(S_{n+1}-S_{n}\geq i), \quad & j=0.\end{cases}
 \eeqnn
The basic assumption that $S$ is right or left continuous implies that
 \beqnn
p_{i,j}=\begin{cases} \bfP(S_{n+1}-S_n=i-j), & j\geq i-1, i\geq 1,
\\ \bfP(S_{n+1}-S_{n}=0)+\bfP(S_{n+1}-S_{n}=1), \quad & j=0, i=0,
\\ \bfP(S_{n+1}-S_n=-j),\quad & j>0, i=0,
\\ 0, \quad&\text{otherwise}.\end{cases}
 \eeqnn
When $S$ is right continuous, the transition probability $p_{i,j}$ is given by (\ref{rc-1}) with $p_{_{-1}}=q$. Similarly, when $S$ is left continuous, the transition probability $p_{i,j}$ is given by (\ref{lc-1}) with $p_{_{-1}}=q$.

Let $ L_n^0(\bar{S})$ be the occupation time of $\bar{S}$ at the site $0$ from time $1$ up to time $n$, that is,
 $$L_{_0}^{0}(\bar{S})=0,~~\text{and}~~L_n^0(\bar{S})=\sum_{k=1}^n{\bf 1}_{\{\bar{S}_k=0\}}~~\text{for}~~n\geq 1.$$
It is easy to see that for every $n\geq 0$,
 \beqlb\label{form1}
 A_n=L_n^0(\bar{S}).
 \eeqlb
Let $\bar{\tau}_{_1}:=\inf\{n>0, \bar{S}_n=0\}$ and $\bar{\tau}_{k+1}:=\inf\{n>\tau_{k}, \bar{S}_n=0\}$ for $k\geq 1$. (\ref{form1}) suggests that  $A_n=L_n^0(\bar{S})=\sup\{k\geq 1, \bar{\tau}_k\leq n\}$. The Markov property indicates that $\bar{\tau}_{_1}$ and
 $\bar{\tau}_{k+1}-\bar{\tau}_k,\; k\geq 1$
are i.i.d.

\medskip

We next prove the LDP for $A_n$.

\noindent{\bf Proof of Theorem \ref{ldp}}. Let $\{Y_i,~i\geq 1\}$ be a sequence of i.i.d. random variables with the same distribution as $\bar{\tau}_{_1}$. Then we have $$\bfP\left(\sum\limits _{i=1}^{\lceil x n\rceil} Y_i\leq n\right)\leq \bfP(A_n\geq x n)\leq \bfP\left(\sum\limits_{i=1}^{\lfloor x n\rfloor} Y_i\leq n\right)$$
for any $0<x\leq 1$, where $\lceil a\rceil$ and $\lfloor a\rfloor$ denote the minimal integer no smaller than $a$ and the maximal integer no larger than $a$, respectively.

When $S$ is right continuous, since $\bar{S}_{_0}=0$ and $Y\stackrel{d}{=}\bar{\tau}_{_1}$, we can get from Lemma \ref{lemrc-1} that $\bfE(\e^{\lambda Y })=\Lambda_r(\lambda)$  for any $\lambda<0$, and that $\bfE(Y)=+\infty$. Applying Cram\'{e}r's Theorem \cite[Theorem 2.2.3]{DZ98}, Lemma \ref{lemlambda} and Lemma \ref{lemlimit}, we obtain
 \beqnn
  \lim_{n\to +\infty}\frac{1}{n}\ln \bfP\left(\sum _{i=1}^{n} Y_i\leq x n\right)=-\Lambda_r^*(x).
 \eeqnn
Similarly, when $S$ is left continuous,
  \beqnn
  \lim_{n\to +\infty}\frac{1}{n}\ln \bfP\left(\sum _{i=1}^{n} Y_i\leq x n\right)=-\Lambda_l^*(x).
 \eeqnn
The rest is the same as the proof of Theorem 2 in \cite{GZ98}. So we omit the details.\qed

\section{The proof of MDP}

In this section, we will first prove the MDP for $A_n$ under the assumption (H). Then we provide some sufficient conditions for the assumption (H). Based on these sufficient conditions, we can directly get Corollaries \ref{thmmdp-2} and \ref{thmmdp-3}.

\medskip

The following lemma is a special presentation of Chen \cite[Theorem 2]{C01} in our case.

\begin{lemma}\label{c01} Suppose that there is a non-decreasing positive function $a(t)$ on $[1, +\infty)$ such that $a(t)\uparrow\infty$ and
 $\displaystyle{\lim_{n\to +\infty}\frac{1}{a(n)}\sum_{k=1}^n \bfP(\bar{S}_n=0|\bar{S}_{_0}=0)=1}$,
and there exists $p\in[0, 1)$ such that $\lim\limits_{\lambda\to +\infty}a(\lambda t)/a(\lambda)=t^p$ for every $t>0$. Furthermore, let $\{b_n\}$ be a positive sequence satisfying $b_n\to +\infty$ and $b_n/n\to 0$ as $n\to +\infty$. Then
 $$\lim_{n\to +\infty}\frac{1}{b_n}\ln\bfP\left(\sum_{k=1}^n{\bf 1}_{\{\bar{S}_k=0\}}>\lambda a\left(\frac{n}{b_n}\right)b_n\right)=-(1-p)\left(\frac{p^p\lambda}{\Gamma(p+1)}\right)^{(1-p)^{-1}}.$$
\end{lemma}

\medskip

Now we provide the proof of the MDP for $A_n$.

\medskip

\noindent {\bf Proof of Theorem \ref{thmmdp-1}} (1) By some basic theory in Markov chains, we can get
 $$\sum_{n=0}^{+\infty} \bfP(\bar{S}_n=0|\bar{S}_{_0}=0)s^n=\frac{1}{1-f_{0}(s)}$$
for any $s\in[0, 1)$. Therefore, from Lemma \ref{mdplem-1}, we know that as $s\to 1-$,
 $$\sum_{n=0}^{+\infty} \bfP(\bar{S}_n=0|\bar{S}_{_0}=0)s^n\sim \frac{(1-\alpha)c}{q}(1-s)^{\alpha-1},$$
where $q=p_{_{-1}}$ in the setting of Theorem \ref{thmmdp-1}. By Tauberian's Theorem \cite[Page 445, Theorem 2]{Feller71}, we know that
 \beqlb\label{mdp-4}
 \lim_{n\to +\infty}\frac{\sum\limits_{k=0}^n \bfP(\bar{S}_n=0|\bar{S}_{_0}=0)}{n^{1-\alpha}}=\frac{(1-\alpha)c}{q\Gamma(2-\alpha)}.
\eeqlb
 Note that $A_n=\sum\limits_{k=1}^n{\bf 1}_{\{\bar{S}_k=0\}}$. From (\ref{mdp-4}), we know that Lemma \ref{c01} is fulfilled for $$p=1-\alpha,~~a(t)=\frac{(1-\alpha)c}{q\Gamma(2-\alpha)}t^{1-\alpha}~~\text{and}~~b_n=c_n.$$ Consequently,
 \begin{align*}
&\lim_{n\to +\infty} \frac{1}{c_n}\ln \bfP(A_n\geq x n^{1-\alpha} c_n^\alpha)=\lim_{n\to +\infty} \frac{1}{c_n}\ln \bfP(A_n\geq x \left(\frac{n}{c_n}\right)^{1-\alpha} c_n)
\\=&\lim_{n\to +\infty} \frac{1}{c_n}\ln \bfP(A_n\geq \frac{x\Gamma(2-\alpha)q}{(1-\alpha)c} a(n/c_n) c_n)=-\alpha\left(\frac{xq}{(1-\alpha)^\alpha c}\right)^{1/\alpha}=-\frac{\alpha}{1-\alpha}\left(\frac{q}{c}\right)^{1/\alpha}x^{1/\alpha}.
 \end{align*}

 \medskip

(2) In this case, similarly we know that Lemma \ref{c01} is fulfilled for $$p=\alpha,~~a(t)=\frac{1}{(1-\alpha)c\Gamma(1+\alpha)}t^{\alpha}~~\text{and}~~b_n=c_n.$$
Therefore,
\begin{align*}
&\lim_{n\to +\infty} \frac{1}{c_n}\ln \bfP\left(A_n\geq x n^{\alpha} c_n^{1-\alpha}\right)=\lim_{n\to +\infty} \frac{1}{c_n}\ln \bfP\left(A_n\geq x \left(\frac{n}{c_n}\right)^{\alpha} c_n\right)\\
=&\lim_{n\to +\infty} \frac{1}{c_n}\ln \bfP(A_n\geq {x\Gamma(1+\alpha)(1-\alpha)c}a(n/c_n) c_n)=-\left[c(1-\alpha)^{2-\alpha}\alpha^{\alpha}\right]^{1/(1-\alpha)}x^{1/(1-\alpha)}.
\end{align*}
The proof is now completed. \qed

%
%
%

\begin{remark}\label{remark1}
 From Lemma \ref{lemma1}, we know that when $\sqrt{\phi^{\prime\prime}(1)}=\sigma<\infty$, Assumption (H) always holds for $\alpha=1/2$ and $c=\sqrt{2}\sigma$.
\end{remark}

\medskip

For the case of $\phi^{\prime\prime}(1)=\infty$, we have the following specific result.

\begin{lemma}\label{lemstable}
If $\phi(s)=s+\frac{\gamma}{1+\beta}(1-s)^{1+\beta}$ for some $\beta\in(0, 1)$ and $\gamma\in(0, 1)$, then Assumption (H) holds for $c=\gamma^{1/(1+\beta)}(1+\beta)^{\beta/(1+\beta)}$ and $\alpha=\beta/(1+\beta)$.
\end{lemma}

\begin{proof} From $h(s)=s\phi(h(s))$, we have $h(s)=sh(s)+\frac{s\gamma}{1+\beta}(1-h(s))^{1+\beta}$, which implies
 \beqlb\label{lemstable-1}
 \frac{1-s}{1-h(s)}=M(s)(1-h(s))^{\beta},
 \eeqlb
where $M(s)=s\gamma/[(1+\beta)h(s)]$. Therefore, for any $s\in(0, 1)$,
 \begin{align*}
 1-s\phi^\prime(h(s))&=1-s(1-\gamma(1-h(s))^\beta)=1-s+\frac{s\gamma(1-h(s))^{1+\beta}}{1-h(s)}
  \\&=1-s+\frac{(1+\beta)h(s)(1-s)}{1-h(s)}=(1+\beta h(s))\frac{1-s}{1-h(s)}.
 \end{align*}
Using (\ref{lemstable-1}) repeatedly, we get that for any $k\geq 1$,
  \beqnn
 &&1-s\phi^\prime(h(s))=(1+\beta h(s))M(s)(1-h(s))^\beta=(1+\beta h(s))M(s)\left(\frac{1-h(s)}{1-s}\right)^\beta (1-s)^\beta
 \\&&\qquad=(1+\beta h(s))M(s)M(s)^{-\beta}(1-h(s))^{-\beta^2}(1-s)^\beta
 \\&&\qquad=(1+\beta h(s))M(s)M(s)^{-\beta}\left(\frac{1-h(s)}{1-s}\right)^{-\beta^2}(1-s)^\beta(1-s)^{-\beta^2}
 \\&&\qquad=(1+\beta h(s))M(s)M(s)^{-\beta}M(s)^{\beta^2}(1-h(s))^{\beta^3}(1-s)^\beta(1-s)^{-\beta^2}
 \\&&\qquad=\cdots=(1+\beta h(s))M(s)^{1+\cdots+(-\beta)^k}(1-s)^{\beta+\cdots+(-1)^{k+1}\beta^k}(1-h(s))^{(-1)^k\beta^{k+1}}.
 \eeqnn
Since $0<\beta<1$ and $0<1-h(s)<1$ for all $s\in(0, 1)$, by letting $k\to +\infty$, we have
 $$M(s)^{1-\beta+\cdots+(-1)^k\beta^k}\to M(s)^{1/(1+\beta)}, \;\; (1-h(s))^{(-1)^k\beta^{k+1}}\to 1 $$
$$\text{and}~~(1-s)^{\beta-\beta^2+\cdots+(-1)^{k+1}\beta^k}\to (1-s)^{\beta/(1+\beta)}.$$
Consequently, $\displaystyle{1-s\phi^\prime(h(s))=(1+\beta h(s))M(s)^{1/(1+\beta)}(1-s)^{\beta/(1+\beta)}}$, which implies that
  \beqnn
 \lim_{s\to 1-}\frac{1-s\phi^\prime(h(s))}{(1-s)^{\beta/(1+\beta)}}=\lim_{s\to 1-}(1+\beta h(s))M(s)^{1/(1+\beta)}=\gamma^{1/(1+\beta)}(1+\beta)^{\beta/(1+\beta)},
 \eeqnn
here we use the fact that $h(s)\to 1$ and $M(s)\to\frac{\gamma}{1+\beta}$ as $s\to 1-$. Therefore, Assumption (H) holds for $c=\gamma^{1/(1+\beta)}(1+\beta)^{\beta/(1+\beta)}$ and $\alpha=\beta/(1+\beta)$.\qed
\end{proof}

From Theorem \ref{thmmdp-1}, Remark \ref{remark1} and Lemma \ref{lemstable}, we can get the Corollaries \ref{thmmdp-2} and \ref{thmmdp-3}, where the fact $q=\gamma/(1+\beta)$ is used for the latter case. The details are omitted.\qed

\section{Concluding remarks}

In this paper, we prove the large and the moderate deviations principle for two kinds of non--nearest neighbor random walks, that is, the left continuous and the right continuous random walks. As implied by our main results~(Theorem \ref{ldp} and Theorem \ref{thmmdp-1}), the form of the asymptotic behavior is different among the left continuous case, the right continuous case, and the nearest--neighbor case. This implies that the traditional method by utilizing the strong invariance principle and relating random walks with Brownian motions may not work for the current cases. Instead, the new approach of linking random walks to some queueing models helps to overcome the above difficulties.

One of the future direction is to extend the results to more general kinds of transition probabilities~(beyond the left continuous and the right continuous setting). However, the current approach may fail since the relation with the queueing models~(as displayed in Section 4) may become invalid for the more general cases. We may think about other estimating approaches in the future.

Another interesting problem is to consider the high dimensional case. This may be much more difficult. Very recently, Godr\`{e}che and Luck \cite{GL2021} considered the two--dimensional case, and they only got the law of large numbers for the nearest--neighbor case.

\bigskip

\textbf{Acknowledgment.} We would like to thank the anonymous referee for the useful suggestions,
which help a lot to improve the manuscript. The project was supported by the National Natural Science Foundation of China~(Grant No. 11671145) and the Science and Technology Commission of Shanghai Municipality~(Grant No. 18dz2271000)

\end{document}